\documentclass{article}%
\usepackage{amsfonts}
\usepackage{mathrsfs}
\usepackage{amsmath}
\usepackage{amssymb}
\usepackage{enumerate}
\usepackage{graphicx}%
\numberwithin{equation}{section}
\newtheorem{theorem}{Theorem}[section]

\newtheorem{corollary}[theorem]{Corollary}

\newtheorem{definition}[theorem]{Definition}

\newtheorem{lemma}[theorem]{Lemma}

\newtheorem{proposition}[theorem]{Proposition}
\newtheorem{remark}[theorem]{Remark}

\newenvironment{proof}[1][Proof]{\noindent\textbf{#1.} }{\ \rule{0.5em}{0.5em}}

\usepackage{cite}
\begin{document}

\begin{center}
\Large \textbf{ Holmstedt's formula for the $K$-functional: the limit case $\theta_0=\theta_1$    \\}
\vskip0.5cm
\large{\textbf{Irshaad Ahmed}${}^{1}$,   \textbf{Alberto  Fiorenza${}^{2}$}} and  \textbf{Amiran Gogatishvili${}^{3}$}\\
\vskip0.5cm
\small{${}^{1}$Department of Mathematics, Sukkur IBA University, Sukkur, Pakistan.\\irshaad.ahmed@iba-suk.edu.pk

\small${}^{2}$Universit\'a di Napoli Federico II, Dipartimento di
Architettura, via Monteoliveto, 3, 80134 - Napoli,   Italy and
Consiglio Nazionale delle Ricerche, Istituto per le Applicazioni del
Calcolo ``Mauro Picone", Sezione di Napoli, via Pietro Castellino,
111, 80131 - Napoli, Italy.\\fiorenza@unina.it

\small${}^{3}$Institute of  Mathematics of the   Czech Academy of Sciences - \v Zitn\'a, 115 67 Prague 1,  Czech Republic.\\gogatish@math.cas.cz }
\end{center}

\begin{abstract} We consider $K$-interpolation spaces involving  slowly varying functions, and derive necessary and sufficient conditions for a Holmstedt-type formula  to be held in the limiting case $\theta_0=\theta_1\in\{0,1\}.$ We also study the case $\theta_0=\theta_1\in (0,1).$  Applications are given to Lorentz-Karamata spaces, generalized gamma spaces and Besov spaces.
\end{abstract}



\section{Introduction}
Let $(A_0,A_1)$  be a compatible couple of quasi-normed spaces.   For each $ f\in A_0+A_1$ and $t>0$, the Peetre's  $K$-functional is defined by
\begin{eqnarray*}
K(t,f)&=&K(t,f;A_0,A_1)\\
&=&\inf\{\|f_0\|_{A_0}+t\|f_1\|_{A_1}:\;f_0 \in A_0, \; f_1 \in A_1,\; f=f_0+f_1\}.
\end{eqnarray*}
Let $0<q\leq \infty, $ $0\leq \theta\leq1$, and let $b$ be a slowly varying function on $(0,\infty)$.   The $K$-interpolation space ${\bar{A}}_{\theta,q;b}=(A_{0},A_{1})_{\theta,q;b}$ is formed of those $f\in A_0+A_1$ for which the quasi-norm
 $$
\|f\|_{\bar{A}_{\theta,q;b}}=\|t^{-\theta-1/q}b(t)K(t,f)\|_{q, (0,\infty)}
$$
is finite; see \cite{GOT}. If $b\equiv 1$ and $(\theta, q)\in([0,1]\times[1,\infty])\setminus(\{0,1\}\times[1,\infty])$, then we recover the classical real interpolation spaces  ${\bar{A}}_{\theta,q}$ (see \cite{BL, T, BS}).  Let $0<q_0,q_1\leq \infty.$ The celebrated classical   Holmstedt's formula states that, for all for $f\in A_0+A_1$ and for all $t>0$,  we have
\begin{eqnarray*}
K(t^{\theta_1-\theta_0},f;{\bar{A}}_{\theta_0,q_0}, {\bar{A}}_{\theta_1,q_1})&\approx&\|u^{-\theta_0-1/{q_0}}K(u,f)\|_{q_0, (0,t)}\\
&&+t^{\theta_1-\theta_0}\|u^{-\theta_1-1/{q_1}}K(u,f)\|_{q_1, (t,\infty)},
\end{eqnarray*}
provided $0<\theta_0< \theta_1<1$ (see \cite[Theorem 2.1]{Ho}). In the limiting case $\theta=0$ or $\theta =1$, the classical space ${\bar{A}}_{\theta,q}$ contains only zero element unless $q=\infty.$ However, the limiting  $K$-interpolation spaces ${\bar{A}}_{0, q;b}$ and  ${\bar{A}}_{1, q;b}$ do make sense (for all $q\in(0,\infty]$) under appropriate conditions on $b$. In the non-limiting case $0<\theta_0<\theta_1<1$, we have the following straightforward extension of the classical Holmstedt's formula (see\cite[Theorem 3.1]{GOT}):
\begin{eqnarray*}
K(w(t),f;{\bar{A}}_{\theta_0,q_0;b_0}, {\bar{A}}_{\theta_1,q_1;b_1})&\approx&\|u^{-\theta_0-1/{q_0}}b_0(u)K(u,f)\|_{q_0, (0,t)}\\
&&+w(t)\|u^{-\theta_1-1/{q_1}}b_1(u)K(u,f)\|_{q_1, (t,\infty)},
\end{eqnarray*}
where $b_0$ and $b_1$ are slowly varying functions and $w(t)=t^{\theta_1-\theta_0}b_0(t)/b_1(t).$ The limiting case when $\theta_0=0$ and $\theta_1=1$ is contained in \cite[Example 5]{AKA}.  However, the limiting case $\theta_0=\theta_1\in [0,1]$ still remains open for general slowly varying functions $b_0$ and $b_1$. The main goal of the current  paper is to fill this gap.\\

Let us describe our main results. To this end,  let
$$
\rho(t)=\frac{\|u^{-1/q_0}b_0(u)\|_{q_0,(t,\infty)}}{\|u^{-1/q_1}b_1(u)\|_{q_1,(t,\infty)}},\;\; t>0
$$
and, for each $\epsilon>0$, put
$$
\rho_{\epsilon}(t)=\frac{\|u^{-1/q_0}b_0(u)\|_{q_0,(t,\infty)}^{1+\epsilon}}{\|u^{-1/q_1}b_1(u)\|_{q_1,(t,\infty)}}.
$$

 In the limiting case $\theta_0=\theta_1=0$, there are two distinguishing    cases:  $q_0\neq q_1$ and $q_0=q_1$. In the case when  $q_0\neq q_1$,  we establish that the following version of Holmstedt's formula
\begin{eqnarray*}
K(\rho(t),f;{\bar{A}}_{0,q_0;b_0}, {\bar{A}}_{0,q_1;b_1})&\approx&\|u^{-1/{q_0}}b_0(u)K(u,f)\|_{q_0, (0,t)}\\
&&+\rho(t)\|u^{-1/{q_1}}b_1(u)K(u,f)\|_{q_1, (t,\infty)},
\end{eqnarray*}
holds for   all for $f\in A_0+A_1$ and for all $t>0$ provided  the following condition is met:  $\rho_{\epsilon}$ is equivalent to a non-decreasing function for some $\epsilon >0.$ This condition also turns out to be necessary  under the additional assumption that the given couple $(A_0, A_1)$ is  $K$-surjective  (see Definition \ref{DKS} below).   When $q_0=q_1<\infty$, the previous estimate holds if  $\rho$ is increasing  and the couple $(A_0, A_1)$ is  $K$-surjective, and when $q_0=q_1=\infty$ the previous estimate holds under the natural condition that $\rho$ is increasing. The corresponding Holmstedt's formulae for  the symmetric counterpart limiting case $\theta_0=\theta_1=1$  follows immediately, by the usual symmetry argument, from the limiting case $\theta_0=\theta_1=0$.  Finally, in the limiting case $ \theta_0=\theta_1\in (0,1)$ we  further have   two distinguishing    cases: while in the case  $q_0=q_1$, a version of Holmstedt's formula does exist, there   exists no analogue of Holmstedt's formula  in the case  $q_0\neq q_1$  if  the given couple $(A_0, A_1)$ is  $K$-surjective. \\

The reader is referred to recent works \cite{AEEK, AKA, FS1, FS2, FS3, FS4, FS5, FS6,  DFS1, D2, D3} for other generalized versions of Holmstedt's formula.\\

The paper is organised as follows. All the Holmstedt's formulae mentioned above are contained in Section 3. The necessary background is collected in Section 2. Section 4 contains the reiteration formulae, and some concrete examples of these reiteration formulae are included in the final section 5.

\section{Background material}\label{due}
\subsection{Notation}
We write $A\lesssim B$ or $B\gtrsim A$  for two non-negative quantities $A$ and $B$ to mean that $A\leq c B$ for some positive constant $c$ which is independent of appropriate parameters involved in $A$ and $B$. If both the estimates   $A \lesssim B$ and  $B\lesssim A$ hold, we simply put $A\approx B.$ We let $\|\cdot\|_{q, (a,b)}$ denote the standard $L^q$-quasi-norm on an interval $(a,b)\subset \mathbb{ R}$. We write $X\hookrightarrow Y$ for two quasi-normed spaces $X$ and $Y$ to mean that $X$ is continuously embedded in $Y.$
\subsection{Slowly varying functions}
Let $b:(0,\infty)\rightarrow (0,\infty)$  be a Lebesgue measurable function. Following \cite{GOT}, we say $b$ is slowly varying on $(0,\infty)$ if for every  $\varepsilon>0,$  there  are  positive  functions $g_{\varepsilon}$ and $g_{-\varepsilon}$ on $(0,\infty)$ such that $g_{\varepsilon}$ is non-decreasing and  $g_{-\varepsilon}$ is non-increasing, and we have
$$
t^{\varepsilon}b(t)\approx g_{\varepsilon}(t)\; \; \text{and}\;\;  t^{-\varepsilon}b(t)\approx g_{-\varepsilon}(t)\;\; \text{for all }\; t\in (0,\infty).
$$

We denote the class of all slowly varying functions by $SV.$ Let $\mathbb{A}=(\alpha_0,\alpha_\infty)\in \mathbb{R}^2.$ Define
$$\ell^\mathbb{A}(t)=\left\{
\begin{array}
[c]{cc}%
(1-\ln t)^{\alpha_0}, & 0<t\leq 1,\\
& \\
(1+\ln t)^{\alpha_\infty}, &
t>1,
\end{array}
\right.
$$
 Then $\ell^\mathbb{A}\in SV.$ In addition, the class $SV$ contains compositions of appropriate log-functions, $\exp |\log t|^\alpha$ with $\alpha\in (0,1)$, etc.

We collect in next Proposition some elementary properties of slowly varying functions, which  will be used in the sequel time and again without explicit mention. The proofs of these assertions can be carried out as in \cite[Lemma $2.1$]{GOT} or \cite[Proposition $3.4.33$]{EE}.
\begin{proposition}\label{svp1} Given  $b$, $b_1$, $ b_2\in SV$, the following assertions hold:

\begin{enumerate}[(i)]
    \item  $b_1b_2\in SV$ and  $b^{r} \in SV$ for each $ r\in \mathbb{R}.$
    \item If $g(t)\approx h(t),\; t>0$, then $b(g(t))\approx b(h(t)),\; t>0.$

    \item  If $\alpha>0$, then

$$\|u^{\alpha-1}b(u)\|_{1,(0,t)}\approx t^{\alpha}b(t),\quad t>0.$$

\item If $\alpha>0$, then
$$\|u^{-\alpha-1}b(u)\|_{1,(t,\infty)}\approx t^{-\alpha}b(t),\quad t>0.$$

\item  Assume that $$\|u^{-1}b(u)\|_{1,(1,\infty)}<\infty,$$
and set
$$\tilde{b}(t)=\|u^{-1}b(u)\|_{1,(t,\infty)},\quad t>0.$$  Then $\tilde{b}\in SV,$ and $b(t)\lesssim \tilde{b}(t),\; t>0.$
\end{enumerate}

\end{proposition}

\subsection{$K$-interpolation spaces}

Let $A_0$ and $A_1$  be two quasi-normed spaces. We say $(A_0,A_1)$ is a  compatible  couple if   $A_0$ and $A_1$ are continuously embedded in the same Hausdorff topological vector space.  For each $ f\in A_0+A_1$ and $t>0$, the  $K$-functional is defined by
\begin{eqnarray*}
K(t,f)&=&K(t,f;A_0,A_1)\\
&=&\inf\{\|f_0\|_{A_0}+t\|f_1\|_{A_1}:\;f_0 \in A_0, \; f_1 \in A_1,\; f=f_0+f_1\}.
\end{eqnarray*}
Note that $K(t,f)$ is, as a function of $t$, non-decreasing on $(0,\infty),$ while $K(t,f)/t$ is,  as a function of $t$, non-increasing on $(0,\infty),$

Let $0<q\leq \infty, $ $0\leq \theta\leq1$, and let $b \in SV $.   The $K$-interpolation space ${\bar{A}}_{\theta,q;b}=(A_{0},A_{1})_{\theta,q;b}$ is formed of those $f\in A_0+A_1$ for which the quasi-norm
 $$
\|f\|_{\bar{A}_{\theta,q;b}}=\|t^{-\theta-1/q}b(t)K(t,f;A_0,A_1)\|_{q, (0,\infty)}
$$
is finite; see \cite{GOT}. If $b=\ell^\mathbb{A},$ then we obtain the $K$-interpolation spaces $\bar{A}_{\theta,q;\mathbb{A}}$ considered in \cite{EO}  and \cite{EOP}. If $\mathbb{A}=(0,0)$ and $(\theta, q)\in([0,1]\times[1,\infty])\setminus(\{0,1\}\times[1,\infty])$, then we recover the classical $K$-interpolation spaces  ${\bar{A}}_{\theta,q}$ (see \cite{BL, T, BS}).

It is not hard to check that for $\theta \in (0,1)$  the spaces $\bar{A}_{\theta,q;b}$  are intermediate, without any condition on $b$ and $q$, for the couple $(A_0,A_1)$, that is,
$$A_0\cap A_1\hookrightarrow \bar{A}_{\theta,q;b} \hookrightarrow  A_0+A_1.$$ However, while   working  with the limiting spaces ${\bar{A}}_{0,q;b}$ and ${\bar{A}}_{1,q;b}$, we have to impose an appropriate condition on $b$ and $q.$ For convenience, let us introduce two notations.  We say $b\in SV_{0,q}$ if $b\in SV$ and
\begin{equation*}\label{ntc1}
\|u^{-1/q}b(u)\|_{q, (1,\infty)}< \infty.
 \end{equation*}
And we say $b\in SV_{1,q}$ if  $b(1/t)\in SV_{0,q}.$  For  $b\in SV_{0,q}$ (or  $b\in SV_{1,q}$),   the space ${\bar{A}}_{0,q;b}$ (or ${\bar{A}}_{1,q;b}$) is intermediate for the couple $(A_0,A_1)$ (see \cite [Proposition 2.5]{GOT}).

\subsection{Weighted  inequalities}
First let us recall that  a function $\phi:(0,\infty) \to (0,\infty)$ is called  quasi-concave if both $\phi$ and $t\phi(1/t)$ are non-decreasing.
\begin{theorem}\label{TQFi}\cite[Theorem 5.1]{GP}
Let $0<p,q<\infty,$ and  let $v$ and $w$ be positive functions on $(0,\infty).$ Consider the inequality
\begin{equation}\label{TQFe1}
\left(\int_{0}^{\infty}\left[ h(s)w(s)\right]^{q}\frac{ds}{s}\right)^{1/q}\leq C \left(\int_{0}^{\infty}\left[ h(s)v(s)\right]^{p}\frac{ds}{s}\right)^{1/p}.
\end{equation}

(a) Let $0<p\leq q<\infty.$ Then the inequality (\ref{TQFe1}) holds  for all quasi-concave functions $h$ on $(0,\infty)$ if and only if
$$
A_1=\sup\limits_{x>0}\frac{\left(\int_{0}^{x}s^qw^{q}(s)\frac{ds}{s}+x^q\int_{x}^{\infty}w^{q}(s)\frac{ds}{s}\right)^{1/q}}{\left(\int_{0}^{x}s^pv^{p}(s)\frac{ds}{s}+x^p\int_{x}^{\infty}v^p(s)\frac{ds}{s}\right)^{1/p}}<\infty.
$$
Moreover,   $C=A_1$ is the best constant.

(b) Let $0<q < p<\infty.$  Then the inequality (\ref{TQFe1}) holds  for all quasi-concave functions $h$ on $(0,\infty)$ if and only if
$$
A_2=\left(\int_{0}^{\infty}\frac{\left(\int_{0}^{x}s^qw^{q}(s)\frac{ds}{s}+x^q\int_{x}^{\infty}w^{q}(s)\frac{ds}{s}\right)^{\frac{q}{p-q}}}{\left(\int_{0}^{x}s^pv^{p}(s)\frac{ds}{s}+x^p\int_{x}^{\infty}v(s)^{p}\frac{ds}{s}\right)^{\frac{q}{p-q}}}x^qw^q(x)\frac{dx}{x}\right)^{1/q-1/p}<\infty.
$$
Moreover,   $C=A_2$ is the best constant.
\end{theorem}
\begin{corollary}\label{C1QCF}
Let $0<p,q<\infty,$ and  let $v\in SV_{0,p}$ and $w\in SV_{0,q}.$

(a) Let $0<p\leq q<\infty.$ Then the inequality (\ref{TQFe1}) holds  for all quasi-concave functions $h$ on $(0,\infty)$ if and only if

$$
A_3=\sup\limits_{x>0}\frac{\left(\int_{x}^{\infty}w^{q}(s)\frac{ds}{s}\right)^{1/q}}{\left(\int_{x}^{\infty}v^p(s)\frac{ds}{s}\right)^{1/p}}<\infty.
$$
Moreover,   $C=A_3$ is the best constant.

(b) Let $0<q < p<\infty.$ Then the inequality (\ref{TQFe1}) holds  for all quasi-concave functions $h$ on $(0,\infty)$ if and only if

$$
A_4=\left(\int_{0}^{\infty}\frac{\left(\int_{x}^{\infty}w^{q}(s)\frac{ds}{s}\right)^{\frac{q}{p-q}}}{\left(\int_{x}^{\infty}v(s)^{p}\frac{ds}{s}\right)^{\frac{q}{p-q}}}w^q(x)\frac{dx}{x}\right)^{1/q-1/p}<\infty.
$$
Moreover,   $C=A_4$ is the best constant.
\end{corollary}
\begin{proof} The proof immediately follows from the previous theorem in view of the assertions (iv) and (vi) in Proposition \ref{svp1}.
\end{proof}
\begin{corollary}\label{C2QCF}
Let $0<p,q<\infty,$ and  let $v\in SV_{0,p}$ and $w\in SV_{0,q}.$ Consider the inequality

\begin{equation}\label{C2QCFe1}
\left(\int_{0}^{\infty}\left[ h(s)w(s)\chi_{(0,t)}(s)\right]^{q}\frac{ds}{s}\right)^{1/q}\lesssim \phi(t) \left(\int_{0}^{\infty}\left[ h(s)v(s)\right]^{p}\frac{ds}{s}\right)^{1/p}.
\end{equation}
(a) Let $0<p\leq q<\infty.$ Then the inequality (\ref{C2QCFe1}) holds  for all quasi-concave functions $h$ on $(0,\infty)$ if and only if we have

\begin{equation}\label{C2QCFe2}
\sup\limits_{0<x<t}\frac{\left(\int_{x}^{\infty}w^{q}(s)\frac{ds}{s}\right)^{1/q}}{\left(\int_{x}^{\infty}v^p(s)\frac{ds}{s}\right)^{1/p}}\lesssim \phi(t),\;\; t>0.
\end{equation}

(b) Let $0<q < p<\infty.$ Then the inequality (\ref{C2QCFe1}) holds  for all quasi-concave functions $h$ on $(0,\infty)$ if and only if  we have

\begin{equation}\label{C2QCFe3}
\left(\int_{0}^{t}\frac{\left(\int_{x}^{\infty}w^{q}(s)\frac{ds}{s}\right)^{\frac{q}{p-q}}}{\left(\int_{x}^{\infty}v(s)^{p}\frac{ds}{s}\right)^{\frac{q}{p-q}}}w^q(x)\frac{dx}{x}\right)^{1/q-1/p}\lesssim \phi(t), \;\; t>0.
\end{equation}
\end{corollary}

\begin{corollary}\label{C3QCF}
Let $0<p,q<\infty,$ and  let $v\in SV_{0,p}$ and $w\in SV_{0,q}.$ Consider the inequality

\begin{equation}\label{C3QCFe1}
\left(\int_{0}^{\infty}\left[ h(s)w(s)\chi_{(t,\infty)}(s)\right]^{q}\frac{ds}{s}\right)^{1/q}\lesssim \psi(t) \left(\int_{0}^{\infty}\left[ h(s)v(s)\right]^{p}\frac{ds}{s}\right)^{1/p}.
\end{equation}
(a) Let $0<p\leq q<\infty.$ Then the inequality (\ref{C3QCFe1}) holds  for all quasi-concave functions $h$ on $(0,\infty)$ if and only if we have

\begin{equation}\label{C3QCFe2}
\sup\limits_{x>t}\frac{\left(\int_{x}^{\infty}w^{q}(s)\frac{ds}{s}\right)^{1/q}}{\left(\int_{x}^{\infty}v^p(s)\frac{ds}{s}\right)^{1/p}}\lesssim \psi(t),\;\; t>0.
\end{equation}

(b) Let $0<q < p<\infty.$ Then the inequality (\ref{C3QCFe1}) holds  for all quasi-concave functions $h$ on $(0,\infty)$ if and only if  we have

\begin{equation}\label{C3QCFe3}
\left(\int_{t}^{\infty}\frac{\left(\int_{x}^{\infty}w^{q}(s)\frac{ds}{s}\right)^{\frac{q}{p-q}}}{\left(\int_{x}^{\infty}v(s)^{p}\frac{ds}{s}\right)^{\frac{q}{p-q}}}w^q(x)\frac{dx}{x}\right)^{1/q-1/p}\lesssim \psi(t), \;\; t>0.
\end{equation}
\end{corollary}

\begin{theorem} \label{HET1}\cite[Lemma $3.2$]{AEEK}
Let $1< \alpha < \infty,$ and assume that $w$ and $\phi$  are positive   functions on $(0,\infty)$. Put
\begin{equation*}
v(t)=(w(t))^{1-\alpha}\left(  \phi(t)\int_{t}^{\infty}w(u)du\right)^{\alpha}.
\end{equation*}
Then
\begin{equation*}
\int_{0}^{\infty}\left(  \int_{0}^{t}\phi(u)h(u)du\right)^{\alpha}w(t)dt \lesssim
\int_{0}^{\infty}h^{\alpha}(t)v(t)dt
\end{equation*}
holds for all positive  functions $h$ on $(0,\infty)$.
\end{theorem}
An appropriate change of variable gives us the following variant of the previous theorem.
\begin{theorem} \label{HET2}
Let $1< \alpha < \infty,$ and assume that $w$ and $\phi$  are positive   functions on $(0,\infty)$. Put
\begin{equation*}
v(t)=(w(t))^{1-\alpha}\left(  \phi(t)\int_{0}^{t}w(u)du\right)^{\alpha}.
\end{equation*}
Then
\begin{equation*}
\int_{0}^{\infty}\left(\int_{t}^{\infty}\phi(u)h(u)du\right)^{\alpha}w(t)dt \lesssim
\int_{0}^{\infty}h^{\alpha}(t)v(t)dt
\end{equation*}
holds for all positive  functions $h$ on $(0,\infty)$.
\end{theorem}
\begin{theorem}\label{HET3+}\cite[Lemma $3.3$]{AEEK}
Let $0< \alpha < 1,$ and assume that $w$ and $\phi$  are positive   functions on $(0,\infty)$. Put
\begin{equation*}
v(t)=\phi(t)\left(\int_{0}^{t}\phi(u)du\right)  ^{\alpha-1}\int_{t}^{\infty
}w(u)du.
\end{equation*}
Then
\begin{equation*}\label{hie1}
\int_{0}^{\infty}\left(  \int_{t}^{\infty}\phi(u)h(u)du\right)^{\alpha}w(t)dt \lesssim
\int_{0}^{\infty}h^{\alpha}(t)v(t)dt
\end{equation*}
holds for all positive, non-increasing  functions $h$ on $(0, \infty).$
\end{theorem}
Again an appropriate change of variable gives us the following variant of the previous theorem.
\begin{theorem} \label{HET3}
Let $0< \alpha < 1,$ and assume that $w$ and $\phi$  are positive   functions on $(0,\infty)$. Put
\begin{equation*}
v(t)=\phi(t)\left(  \int_{t}^{\infty}\phi(u)du\right)  ^{\alpha-1}\int_{0}^{t
}w(u)du.
\end{equation*}
Then
\begin{equation*}\label{hie1}
\int_{0}^{\infty}\left(  \int_{t}^{\infty}\phi(u)h(u)du\right)^{\alpha}w(t)dt \lesssim
\int_{0}^{\infty}h^{\alpha}(t)v(t)dt
\end{equation*}
holds for all positive, non-decreasing  functions $h$ on $(0, \infty).$
\end{theorem}
\begin{theorem}\label{hmt1}\cite[Theorem $3.3$ (b)]{HM}
Let $0<\alpha\leq 1.$ Assume that  $w$ and $v$ are non-negative functions on $(0,\infty),$ and  $\psi$ is a non-negative  function on  $(0,\infty)\times (0,\infty).$ Then
\begin{equation}\label{hme1}
\int_{0}^{\infty}\left(\int_{0}^{\infty}\psi(t, u)h(u)du\right)  ^{\alpha}w(t)dt\lesssim
\int_{0}^{\infty}h^{\alpha}(t)v(t)dt%
\end{equation}
holds for all non-negative, non-decreasing functions $h$ on $(0, \infty)$ if and only if

\begin{equation}\label{hme2}
\int_{0}^{\infty}\left(\int_{x}^{\infty}\psi(t, u)du\right)  ^{\alpha}w(t)dt\lesssim
\int_{x}^{\infty}v(t)dt%
\end{equation}
holds for all $x>0.$
\end{theorem}

\section{Holmstedt-type formulae}
This section contains our main results. In order to describe our results, we need the following class of compatible couples.
\begin{definition}\label{DKS}\cite[p. 217]{N} \;We say a compatible couple $(A_0,A_1)$ of quasi-normed spaces is $K$-surjective if for every quasi-concave function $\phi$, there exists $f\in A_0+A_1$ such that $$\phi(t)\approx K(t,f),\;\; t>0.$$
\end{definition}
\subsection{The case $\theta_0=\theta_1=0$}
Let $0<q_j\leq\infty$ and $b_j\in SV_{0,q_j}$ $(j=0,1).$ Put
$$
\rho(t)=\frac{\|u^{-1/q_0}b_0(u)\|_{q_0,(t,\infty)}}{\|u^{-1/q_1}b_1(u)\|_{q_1,(t,\infty)}},
$$
$$I(t,f)=\|u^{-1/q_0}b_0(u)K(u,f;A_0,A_1)\|_{q_0, (0,t)},$$
and
$$ J(t,f)=\|u^{-1/q_1}b_1(u)K(u,f;A_0,A_1)\|_{q_1, (t,\infty)}.$$
Moreover, let  $\epsilon>0$ and set
$$
\rho_{\epsilon}(t)=\frac{\|u^{-1/q_0}b_0(u)\|_{q_0,(t,\infty)}^{1+\epsilon}}{\|u^{-1/q_1}b_1(u)\|_{q_1,(t,\infty)}}.
$$
\begin{theorem}\label{HomET1} Let $0<q_0,q_1 \leq\infty$, $q_0\neq q_1,$  and  $b_j \in SV_{0,q_j}$  (j=0,1). Assume that the following condition is met:
\begin{equation}\label{HE1e1}
 \rho_\epsilon\; \text{is equivalent to a non-decreasing function for some}\; \epsilon>0.
 \end{equation}
Then, for all $f\in A_0+A_1$ and all $t>0,$ we have
\begin{equation}\label{kfe1}
K(\rho(t),f; \bar{A}_{0,q_0; b_0}, \bar{A}_{0,q_1;b_1} ) \approx I(t,f)+ \rho(t) J(t,f).
\end{equation}
Moreover, the condition (\ref{HE1e1}) is also necessary  provided that the given   couple $(A_0,A_1)$ is $K$-surjective.
\end{theorem}

\begin{proof}  First assume that the condition (\ref{HE1e1}) is met.
According to the estimate (2.30) in \cite[Theorem 2.3]{AEEK}, we have the following estimate from below:
\begin{align}\label{kfea1T3}
K(\rho(t),f; \bar{A}_{0,q_0; b_0}, \bar{A}_{0,q_1;b_1} ) &\lesssim  I(t,f)+ \rho(t) J(t,f)+\rho(t)b_1(t)K(t,f)\notag\\
 &\;\;\;\;+ \|u^{-1/q_0}b_0(u)\|_{q_0,(t,\infty)}K(t,f).
\end{align}
Since $t\mapsto K(t,f)$ is non-decreasing, we obtain
\begin{equation}\label{kfea2}
J(t,f)\geq K(t,f)\|u^{-1/q_1}b_1(u)\|_{q_1,(t,\infty)},
\end{equation}
whence we get
\begin{equation}\label{kfea3}
\rho(t)J(t,f)\geq  \|u^{-1/q_0}b_0(u)\|_{q_0,(t,\infty)}K(t,f).
\end{equation}
Moreover, by Proposition \ref{svp1} (vi), (\ref{kfea2}) gives
\begin{equation}\label{kfea4}
J(t,f)\gtrsim K(t,f)b_1(t).
\end{equation}
Now, in view of (\ref{kfea3}) and (\ref{kfea4}), the estimate $``\lesssim"$ in  (\ref{kfe1})  follows from (\ref{kfea1T3}).  In order  to  establish the converse estimate $``\gtrsim"$, it will suffice to show that the following estimates
\begin{equation*}\label{ckt1e1}
I(t,f)\lesssim \|f_0\|_{\bar{A}_{0,q_0; b_0}}+ \rho(t)\|f_1\|_{\bar{A}_{0,q_1; b_1}},
\end{equation*}
and
\begin{equation*}\label{ckt1e2}
\rho(t)J(t,f)\lesssim \|f_0\|_{\bar{A}_{0,q_0; b_0}}+ \rho(t)\|f_1\|_{\bar{A}_{0,q_1; b_1}},
\end{equation*}
hold for  an arbitrary decomposition $f=f_0+f_1$ with $f_j\in A_j$ $ (j=0,1)$.  As  $K(u,f)\lesssim K(u,f_0)+K(u,f_1)$, we have $$I(t,f)\lesssim I(t,f_0)+I(t,f_1),$$ and $$J(t,f)\lesssim J(t,f_0)+J(t,f_1).$$
Clearly,  $I(t,f_0)\leq \|f_0\|_{\bar{A}_{0,q_0; b_0}}$ and $J(t,f_1)\leq \|f_1\|_{\bar{A}_{0,q_1; b_1}}.$  Therefore, it remains to show that
\begin{equation}\label{kfeb1}
I(t,f_1)\lesssim \rho(t)\|f_1\|_{\bar{A}_{0,q_1; b_1}},
\end{equation}
and
\begin{equation}\label{kfeb2}
\rho(t)J(t,f_0)\lesssim \|f_0\|_{\bar{A}_{0,q_0; b_0}}.
\end{equation}
Since
\begin{equation}\label{kfeb3}
\|f_j\|_{\bar{A}_{0,q_j; b_j}}\geq K(x,f_j)\|u^{-1/q_j}b_j(u)\|_{q_j,(x,\infty)},\;\;x>0,
\end{equation}
it follows that  the estimate (\ref{kfeb1}) holds if the following  condition is met:
\begin{equation}\label{kfeb4}
\left\|\frac{x^{-1/q_0}b_0(x)}{\|s^{-1/q_1}b_1(s)\|_{q_1,(x,\infty)}}\right\|_{q_0, (0,t)}\lesssim \rho(t),
\end{equation}
while  the estimate (\ref{kfeb2}) holds if the following  condition is met:
\begin{equation}\label{kfeb5}
\left\|\frac{x^{-1/q_1}b_1(x)}{\|s^{-1/q_0}b_0(s)\|_{q_0,(x,\infty)}}\right\|_{q_1, (t,\infty)}\lesssim 1/\rho(t).
\end{equation}
Next let us derive the estimate (\ref{kfeb4}) using  the condition  (\ref{HE1e1}).  We consider only the case when $q_0<\infty$ since the case $q_0=\infty$ is analogous and easier. Observe that
\begin{eqnarray*}
&&\int_{0}^{t}\frac{b_0^{q_0}(x)}{\left(\int_{x}^{\infty}b_1^{q_1}(s)\frac{ds}{s}\right)^{q_0/q_1}}\frac{dx}{x}\\ &\lesssim&
\frac{\left(\int_{t}^{\infty}b_0^{q_0}(s)\frac{ds}{s}\right)^{1+\epsilon }}{\left(\int_{t}^{\infty}b_1^{q_1}(s)\frac{ds}{s}\right)^{q_0/q_1}}\int_{0}^{t}\left(\int_{x}^{\infty}b_0^{q_0}(s)\frac{ds}{s}\right)^{-\epsilon -1}b_0^{q_0}(x)\frac{dx}{x}\\ &\approx&\frac{\int_{t}^{\infty}b_0^{q_0}(s)\frac{ds}{s}}{\left(\int_{t}^{\infty}b_1^{q_1}(s)\frac{ds}{s}\right)^{q_0/q_1}},
\end{eqnarray*}
whence we get (\ref{kfeb4}). Next we show that  (\ref{kfeb5}) also follows from the condition  (\ref{HE1e1}).  Again we consider only the case when $q_1<\infty.$ This time we observe that
\begin{eqnarray*}
&&\int_{t}^{\infty}\frac{b_1^{q_1}(x)}{\left(\int_{x}^{\infty}b_0^{q_0}(s)\frac{ds}{s}\right)^{q_1/q_0}}\frac{dx}{x}\\ &\lesssim&
\frac{\left(\int_{t}^{\infty}b_1^{q_1}(s)\frac{ds}{s}\right)^{\frac{1}{1+\epsilon} }}{\left(\int_{t}^{\infty}b_0^{q_0}(s)\frac{ds}{s}\right)^{q_1/q_0}}\int_{t}^{\infty}\left(\int_{x}^{\infty}b_1^{q_1}(s)\frac{ds}{s}\right)^{\frac{-1}{1+\epsilon} }b_1^{q_1}(x)\frac{dx}{x}\\ &\approx&\frac{\int_{t}^{\infty}b_1^{q_1}(s)\frac{ds}{s}}{\left(\int_{t}^{\infty}b_0^{q_0}(s)\frac{ds}{s}\right)^{q_1/q_0}},
\end{eqnarray*}
whence we get (\ref{kfeb5}). The proof of the sufficiency of the condition (\ref{HE1e1}) is complete.\\

 Next assume that the estimate (\ref{kfe1}) holds for all $A_0+A_1$ and $t>0. $ We distinguish two cases: $q_1<q_0$ and $q_0<q_1$. First we treat the case $q_1<q_0$. Taking a particular decomposition $f=f+0$, $f\in \bar{A}_{0,q_0; b_0}$ and $0\in \bar{A}_{1,q_1; b_1}$, we obtain
\begin{equation*}\label{nce1}
\rho(t)J(t,f)\lesssim \|f\|_{\bar{A}_{0,q_0; b_0}},
\end{equation*}
from which, according to Corollary \ref{C3QCF} (b), it follows that
\begin{equation}\label{kfeb8Feb}
\left(\int_{t}^{\infty}\frac{\left(\int_{x}^{\infty}b_1^{q_1}(s)\frac{ds}{s}\right)^{\frac{q_1}{q_0-q_1}}}{\left(\int_{x}^{\infty}b_0^{q_0}(s)\frac{ds}{s}\right)^{\frac{q_1}{q_0-q_1}}}b_1^{q_1}(x)\frac{dx}{x}\right)^{1/q_1-1/q_0}\lesssim 1/\rho(t).
\end{equation}
Next we introduce the operator
\begin{equation*}
(Q w)(t)=\int_{t}^{\infty}w(x)\left(\int_{x}^{\infty}b_1^{q_1}(s)\frac{ds}{s}\right)^{-1}b_1^{q_1}(x)\frac{dx}{x}.
\end{equation*}
Set temporarily  $$\sigma=\left[\frac{1}{\rho}\right]^{\frac{q_1q_0}{q_0-q_1}}$$
and
$$B(t)=\int_{t}^{\infty}b_1^{q_1}(s)\frac{ds}{s}.$$
For each $k\in \mathbb{N}$, define $Q^{k+1}=Q(Q^k).$ Then
 \begin{equation}\label{nce2}
(Q^{k+1} w)(t)=\frac{1}{k!}\int_{t}^{\infty}w(x)\ln^k\left[\frac{B(t)}{B(x)}\right]\left[B(x)\right]^{-1}b_1^{q_1}(x)\frac{dx}{x},\;\;k\in \mathbb{N}.
\end{equation}
Moreover we see that (\ref{kfeb8Feb}) translates into
\begin{equation*}
(Q\sigma)(t)\leq c\sigma(t)
\end{equation*}
for some constant $c>0.$ Therefore, we have
\begin{equation*}
(Q^{k+1}\sigma)(t)\leq c^{k+1}\sigma(t),\;\; k\in \mathbb{N},
\end{equation*}
which, in view of (\ref{nce2}),  leads us to
\begin{equation}\label{nce3}
\frac{1}{k!}\int_{t}^{\infty}\sigma(x)\ln^k\left[\frac{B(t)}{B(x)}\right]\left[B(x)\right]^{-1}b_1^{q_1}(x)\frac{dx}{x}\leq c^{k+1}\sigma(t),\;\;k\in\mathbb{N}.
\end{equation}
We choose $\epsilon>0$ such that $\max(\epsilon c,\epsilon)<1.$ Then by (\ref{nce3}), we obtain
\begin{equation*}
\int_{t}^{\infty}\sigma(x)\sum\limits_{k=0}^{\infty}\frac{\ln^k\left[\frac{B(t)}{B(x)}\right]^{\epsilon}}{k!}\left[B(x)\right]^{-1}b_1^{q_1}(x)\frac{dx}{x}\leq c\sum\limits_{k=0}^{\infty}(\epsilon c)^{k}\sigma(t),
\end{equation*}
whence we get
\begin{equation*}
\int_{t}^{\infty}\sigma(x)\left[\frac{B(t)}{B(x)}\right]^{\epsilon}\left[B(x)\right]^{-1}b_1^{q_1}(x)\frac{dx}{x}\leq \frac{c}{1-\epsilon c}\sigma(t),
\end{equation*}
or,
\begin{equation}\label{nce4}
\int_{t}^{\infty}\sigma(x)\left[B(x)\right]^{-\epsilon-1}b_1^{q_1}(x)\frac{dx}{x}\leq \frac{c}{1-\epsilon c}\sigma(t)\left[B(t)\right]^{-\epsilon}.
\end{equation}
Now the converse   estimate
\begin{equation}\label{nce5}
\int_{t}^{\infty}\sigma(x)\left[B(x)\right]^{-\epsilon-1}b_1^{q_1}(x)\frac{dx}{x}\gtrsim\sigma(t)\left[B(t)\right]^{-\epsilon}
\end{equation}
holds as well. Indeed,
\begin{align*}
\int_{t}^{\infty}\sigma(x)\left[B(x)\right]^{-\epsilon-1}b_1^{q_1}(x)\frac{dx}{x}&\gtrsim\left(\int_{t}^{\infty}b_0^{q_0}(s)\frac{ds}{s}\right)^{\frac{q_1}{q_1-q_0}}\int_{t}^{\infty}\left[B(x)\right]^{\frac{q_1q_0}{q_0-q_1}-\epsilon-1}b_1^{q_1}(x)\frac{dx}{x}\\
&\approx\left(\int_{t}^{\infty}b_0^{q_0}(s)\frac{ds}{s}\right)^{\frac{q_1}{q_1-q_0}}\left[B(x)\right]^{\frac{q_1q_0}{q_0-q_1}-\epsilon},
\end{align*}
whence we get (\ref{nce5}). From (\ref{nce4}) and (\ref{nce5}), it follows that
\begin{equation*}\label{nce5}
\int_{t}^{\infty}\sigma(x)\left[B(x)\right]^{-\epsilon-1}b_1^{q_1}(x)\frac{dx}{x}\approx\sigma(t)\left[B(t)\right]^{-\epsilon}
\end{equation*}
which shows that $t\mapsto \sigma(t)\left[B(t)\right]^{-\epsilon}$ is equivalent to a non-increasing function. That is,
$$t\mapsto \left[\frac{1}{\rho(t)}\right]^{\frac{q_1q_0}{q_0-q_1}}\left(\int_{t}^{\infty}b_1^{q_1}(s)\frac{ds}{s}\right)^{-\epsilon} $$
is equivalent to a non-increasing function, or,
$$t\mapsto \rho(t)\left(\int_{t}^{\infty}b_1^{q_1}(s)\frac{ds}{s}\right)^{\frac{\epsilon(q_0-q_1)}{q_0q_1}} $$
is equivalent to a non-decreasing function. It follows that
$$t\mapsto \rho(t)\left(\int_{t}^{\infty}b_0^{q_0}(s)\frac{ds}{s}\right)^{\frac{\epsilon(q_0-q_1)}{q_0(q_0-\epsilon(q_0-q_1))}} $$
is equivalent to a non-decreasing function since
$$\rho(t)\left(\int_{t}^{\infty}b_0^{q_0}(s)\frac{ds}{s}\right)^{\frac{\epsilon(q_0-q_1)}{q_0(q_0-\epsilon(q_0-q_1))}}=\left(\rho(t)\left(\int_{t}^{\infty}b_1^{q_1}(s)\frac{ds}{s}\right)^{\frac{\epsilon(q_0-q_1)}{q_0q_1}}\right)^{\frac{q_0}{q_0-\epsilon(q_0-q_1)}}.$$
Thus the condition  (\ref{HE1e1}) is valid in the case when $q_1<q_0.$ As for the case $q_0<q_1$, we take a particular decomposition $f=0+f$, $0\in \bar{A}_{0,q_0; b_0}$ and $f\in \bar{A}_{1,q_1; b_1}$ to get
\begin{equation*}\label{kfeb1Feb2}
I(t,f)\lesssim \rho(t)\|f\|_{\bar{A}_{0,q_1; b_1}},
\end{equation*}
from which, according to Corollary \ref{C2QCF} (b), it follows that
\begin{equation}\label{kfeb6}
\left(\int_{0}^{t}\frac{\left(\int_{x}^{\infty}b_0^{q_0}(s)\frac{ds}{s}\right)^{\frac{q_0}{q_1-q_0}}}{\left(\int_{x}^{\infty}b_1^{q_1}(s)\frac{ds}{s}\right)^{\frac{q_0}{q_1-q_0}}}b_0^{q_0}(x)\frac{dx}{x}\right)^{1/q_0-1/q_1}\lesssim \rho(t).
\end{equation}
This time we introduce the  operator
\begin{equation*}
(P w)(t)=\int_{0}^{t}w(x)\left(\int_{x}^{\infty}b_0^{q_0}(s)\frac{ds}{s}\right)^{-1}b_1^{q_0}(x)\frac{dx}{x},
\end{equation*}
and using a similar argument as in the case $q_1<q_0$ we can conclude that the condition  (\ref{HE1e1}) is also  valid in the case when $q_0<q_1$. This completes the proof of the theorem.
\end{proof}
\begin{remark}\label{R1}
{\em The argument for the estimate $``\lesssim"$ in  (\ref{kfe1}) also works when $q_0=q_1$. Moreover,  the condition (\ref{HE1e1}) is only required for the estimate $``\gtrsim"$ in  (\ref{kfe1}). }
\end{remark}
\begin{remark}\label{R2}
{\em  The converse estimate in  (\ref{kfeb5}) holds trivially.}
\end{remark}
Next we treat the case $q_0=q_1.$
\begin{theorem}\label{HomET2} Let $0<p\leq\infty$ and $b_j\in SV_{0,p}$. Assume that  $\rho$ is increasing. In the case $p< \infty$, assume additionally that  the given   couple $(A_0,A_1)$ is $K$-surjective.   Then, for all $f\in A_0+A_1$ and all $t>0,$ we have
\begin{equation}\label{kfe2}
K(\rho(t),f; \bar{A}_{0,p; b_0}, \bar{A}_{0,p;b_1} ) \approx I(t,f)+ \rho(t) J(t,f).
\end{equation}
\end{theorem}
\begin{proof}   In view of Remark \ref{R1}, it remains to derive the converse estimate $``\gtrsim"$ in (\ref{kfe2}).  First we consider the case $p=\infty.$ In this case we can assume, without loss of generality, that $b_0$ and $b_1$ are non-increasing functions. Therefore, we have $\rho=b_0/b_1.$ The desired estimate  $``\gtrsim"$ follows from the estimate $(2.35)$ in \cite{AEEK}. Next we turn to the case $0<p<\infty.$ As  in the  previous theorem,  we need to show that the estimates (\ref{kfeb1}) and (\ref{kfeb2})  hold ( with $q_0=q_1=p$) for  an arbitrary decomposition $f=f_0+f_1$ with  $f_j\in A_j$ $ (j=0,1)$. According to Corollary \ref{C2QCF} (a), the  estimate (\ref{kfeb1}) holds if the following condition is met:
\begin{equation}\label{kfeb5+}
\sup\limits_{0<x<t}\rho(x)\lesssim \rho(t),
\end{equation}
while, according to Corollary \ref{C3QCF} (a), the  estimate (\ref{kfeb2}) holds if the following condition is met:
\begin{equation}\label{kfeb7+}
\sup\limits_{x>t}1/\rho(x)\lesssim 1/\rho(t).
\end{equation}
But both (\ref{kfeb5+}) and (\ref{kfeb7+}) hold trivially in view of the fact that $\rho$ is increasing. The proof of the theorem is complete.
\end{proof}\\

\subsection{The case $\theta_0=\theta_1=1$}
The case  $\theta_0=\theta_1=1$ is  symmetric counterpart of the case $\theta_0=\theta_1=1$, and the corresponding estimates  can be derived immediately from the estimates in the previous subsection using the same symmetry argument as in the proof of Theorem 4.3 in \cite{GOT}.\\

In order to formulate the results, we introduce some further notation. Let $0<q_j\leq\infty$ and $b_j\in SV_{1,q_j}$ $(j=0,1).$ Put
$$
\eta(t)=\frac{\|u^{-1/q_0}b_0(u)\|_{q_0,(0,t)}}{\|u^{-1/q_1}b_1(u)\|_{q_1,(0,t)}},
$$
$$I_1(t,f)=\|u^{-1-1/q_0}b_0(u)K(u,f;A_0,A_1)\|_{q_0, (0,t)},$$
and
$$ J_1(t,f)=\|u^{-1-1/q_1}b_1(u)K(u,f;A_0,A_1)\|_{q_1, (t,\infty)}.$$
Moreover, let  $\epsilon>0$ and set
$$
\eta_{\epsilon}(t)=\frac{\|u^{-1/q_0}b_0(u)\|_{q_0,(0,t)}^{1+\epsilon}}{\|u^{-1/q_1}b_1(u)\|_{q_1,(0,t)}}.
$$

\begin{theorem}\label{HomET3} Let $0<q_0,q_1 \leq\infty$, $q_0\neq q_1,$  and  $b_j \in SV_{1,q_j}$  (j=0,1). Assume that the following condition is met:
\begin{equation}\label{HE3e1}
 \eta_\epsilon\; \text{is equivalent to a non-decreasing function for some}\; \epsilon>0.
 \end{equation}
Then, for all $f\in A_0+A_1$ and all $t>0,$ we have
\begin{equation}\label{kf3e1}
K(\eta(t),f; \bar{A}_{1,q_0; b_0}, \bar{A}_{1,q_1;b_1} ) \approx I_1(t,f)+ \rho(t) J_1(t,f).
\end{equation}
Moreover, the condition (\ref{HE3e1}) is also necessary  provided that the given   couple $(A_0,A_1)$ is $K$-surjective.
\end{theorem}

\begin{theorem}\label{HomET4} Let $0<p\leq\infty$ and $b_j\in SV_{1,p}$. Assume that  $\eta$ is increasing. In the case $p< \infty$, assume additionally that  the given   couple $(A_0,A_1)$ is $K$-surjective.   Then, for all $f\in A_0+A_1$ and all $t>0,$ we have
\begin{equation}\label{kf3e2}
K(\eta(t),f; \bar{A}_{1,p; b_0}, \bar{A}_{1,p;b_1} ) \approx I_1(t,f)+ \eta(t) J_1(t,f).
\end{equation}
\end{theorem}

\subsection{ The case $\theta_0=\theta_1\in (0,1)$}
First we treat the case $q_0\neq q_1$.
\begin{theorem}\label{HomET5} Let $0<\theta<1$, $0<q_0\neq q_1 \leq  \infty$, and let $b_0, b_1\in SV$.   Assume that the given couple $(A_0, A_1)$ is $K$-surjective.  Then there exists no positive function $w$ on $(0,\infty)$ such that the following estimate holds
\begin{eqnarray*} K(w(t),f; \bar{A}_{\theta,q_0; b_0}, \bar{A}_{\theta,q_1;b_1} ) &\gtrsim& \left(\int_0^ts^{-\theta q_0}K^{q_0}(s,f)\frac{ds}{s}\right)^{1/q_0}\\&&+ w(t) \left(\int_t^\infty s^{-\theta q_1}K^{q_1}(s,f)\frac{ds}{s}\right)^{1/q_1}
\end{eqnarray*}
for all $f\in A_0+A_1$ and for all $t>0.$
\end{theorem}
\begin{proof} We give the argument only in the case $ 0< q_0<q_1\leq\infty$ since the argument in  the other case $ 0< q_1<q_0\leq\infty$ is similar.  We assume, on the contrary, that there exists such a positive function $w $ on $(0,\infty)$. Taking a  particular decomposition $f=f+0$, $f\in \bar{A}_{\theta,q_0; b_0}$ and $0\in \bar{A}_{\theta,q_1; b_1}$, we obtain
\begin{equation}\label{he5e1}
w(t) \|s^{-\theta-1/q} b_1(s)K(s,f)\|_{q_1, (t,\infty)}\lesssim \|f\|_{\bar{A}_{\theta,q_0; b_0}},
\end{equation}
while taking a particular decomposition $f=0+f$, $0\in \bar{A}_{\theta,q_0; b_0}$ and $f\in \bar{A}_{\theta,q_1; b_1}$, we obtain
\begin{equation}\label{he5e2}
\left(\int_0^t s^{-\theta q_0}b_0^{q_0}(s)K^{q_0}(s,f)\frac{ds}{s}\right)^{1/q_0}\lesssim w(t) \|f\|_{\bar{A}_{\theta,q_1; b_1}}.
\end{equation}
First let $q_1<\infty$. Now, according to Corollary \ref{C3QCF} (a), it follows from (\ref{he5e1}) that
\begin{equation}\label{he5e3}
w(t)\lesssim \frac{b_0(t)}{b_1(t)},\;\;\;\; t>0,
\end{equation}
and while, according to Corollary \ref{C2QCF} (b), it   follows from (\ref{he5e2}) that
\begin{equation}\label{he5e4}
  \left(\int_0^t\left[\frac{b_0(s)}{b_1(s)}\right]^{\frac{q_0q_1}{q_0-q_1}}\frac{ds}{s}\right)^{1/q_1-1/q_0}\lesssim w(t),\;\;\;\; t>0.
\end{equation}
Finally, combining (\ref{he5e3}) and (\ref{he5e4}) yields
\begin{equation*}
  \left(\int_0^t\left[\frac{b_0(s)}{b_1(s)}\right]^{\frac{q_0q_1}{q_0-q_1}}\frac{ds}{s}\right)^{1/q_1-1/q_0}\lesssim  \frac{b_0(t)}{b_1(t)},\;\;\;\; t>0,
\end{equation*}
which is not possible since  $b_0$ and $b_1$ are slowly varying functions. Next we turn to the case $q_1=\infty.$ Choose $q_0<r<\infty$. Then, in view of the well-known embedding $\bar{A}_{\theta,r; b_1}\hookrightarrow \bar{A}_{\theta,\infty; b_1}$, (\ref{he5e2}) gives
\begin{equation*}
\left(\int_0^t s^{-\theta q_0}b_0^{q_0}(s)K^{q_0}(s,f)\frac{ds}{s}\right)^{1/q_0}\lesssim w(t) \|f\|_{\bar{A}_{\theta,r; b_1}},
\end{equation*}
from which, using Corollary \ref{C2QCF} (b), it follows that
\begin{equation}\label{he5e5}
  \left(\int_0^t\left[\frac{b_0(s)}{b_1(s)}\right]^{\frac{q_0r}{q_0-r}}\frac{ds}{s}\right)^{1/r-1/q_0}\lesssim w(t),\;\;\;\; t>0.
  \end{equation}
 Next we choose an $\epsilon>0 $  so that both $\theta+\epsilon$ and $\theta-\epsilon$ lie in the interval $(0,1)$. Next choose a $f\in A_0+A_1$ such that  we have
  $$
K(s,f)=\begin{cases}
			s^{\theta+\epsilon}, & 0<s<t,\\
            t^{2\epsilon}s^{\theta-\epsilon}, & s\geq t.
		 \end{cases}
$$
Then we again get (\ref{he5e3}) from (\ref{he5e1}). This time  combining (\ref{he5e3}) and (\ref{he5e5}) yields
\begin{equation*}
  \left(\int_0^t\left[\frac{b_0(s)}{b_1(s)}\right]^{\frac{q_0r}{q_0-r}}\frac{ds}{s}\right)^{1/r-1/q_0}\lesssim  \frac{b_0(t)}{b_1(t)},\;\;\;\; t>0,
\end{equation*}
which is again not  possible. The proof is complete.
\end{proof}\\

Next in the  case $q_0 = q_1$, a version of Holmstedt's formula does exit.
\begin{theorem}\label{HomET8} Let $0<\theta<1$, $0<q \leq  \infty.$ Let $b_0$ and $b_1$ be slowly varying functions such that $\rho=b_0/b_1$ is non-decreasing. Then for all $f\in A_0+ A_1$ and for all $t>0$, we have
\begin{eqnarray*}
K(\rho(t),f;{\bar{A}}_{\theta,q;b_0}, {\bar{A}}_{\theta,q;b_1})&\approx&\|u^{-\theta-1/{q}}b_0(u)K(u,f)\|_{p, (0,t)}\\
&&+\rho(t)\|u^{-\theta-1/{p}}b_1(u)K(u,f)\|_{p, (t,\infty)}.
\end{eqnarray*}

\end{theorem}
\begin{proof} The proof follows immediately from the estimates (2.30) and (2.35)  in \cite[Theorem 2.3]{AEEK}.
\end{proof}
\section{Reiteration}
\begin{theorem}\label{ReT1} Let $0<q_0,q_1,q <\infty$, $0<\theta<1$,  $b_j \in SV_{0,q_j}$  (j=0,1),  and $b\in SV.$ Assume that $\rho$ is increasing on $(0,\infty)$ with $\lim\limits_{t \to 0^+}\rho(t)=0$ and $\lim\limits_{t \to \infty}\rho(t)=\infty.$  If  $q_0\neq q_1$, assume additionally    that the condition (\ref{HE1e1}) is met, while if $q_0=q_1$,   assume additionally that  the given couple $(A_0, A_1)$ is $K$-surjective and  that the following two-sided estimate holds:
 $$\frac{\rho^\prime(t)}{\rho(t)}\approx\frac{t^{-1}b_1^{q_1}(t)}{\int_{t}^{\infty}b_1^{q_1}(u)\frac{du}{u}}, \;\;t>0.$$
Put
\begin{equation*}
\tilde{b}(t)=[\rho(t)]^{(1-\theta)}b(\rho(t))[b_1(t)]^{q_1/q}\left(\int_{t}^{\infty}b_1^{q_1}(u)\frac{du}{u}\right)^{1/q_1-1/q}.
\end{equation*}
Then
\begin{equation}\label{RTI}
\left(\bar{A}_{0,q_0; b_0}, \bar{A}_{0,q_1;b_1}\right)_{\theta,q;b}= \bar{A}_{0,q;\tilde{b}}.
\end{equation}
\end{theorem}
\begin{proof} We consider only the case $q_0\neq q_1$; the other case $q_0=q_1$ being  similar.  Let $f\in A_0+A_1$, and  set $X=\left(\bar{A}_{0,q_0; b_0}, \bar{A}_{0,q_1;b_1}\right)_{\theta,q;b}$, , $Y=\bar{A}_{0,q;\tilde{b}}$  and
$$\frac{1}{\sigma(t)}=\left\|\frac{x^{-1/q_1}b_1(x)}{\|s^{-1/q_0}b_0(s)\|_{q_0,(x,\infty)}}\right\|_{q_1, (t,\infty)}.$$
Then, in view of (\ref{kfeb5})  along with Remark \ref{R2}, we have $\sigma\approx \rho.$ Thus, by Theorem \ref{HomET1}, we get
\begin{equation*}
\|f\|_{X}^q\approx I_1+I_2,
\end{equation*}
where
\begin{equation*}
I_1=\int_{0}^{\infty}[\sigma(t)]^{-\theta q}b^q(\sigma(t))\left(\int_{0}^{t}b_0^{q_0}(u)K^{q_0}(u,f)\frac{du}{u}\right)^{q/q_0}\frac{\sigma^\prime(t)}{\sigma(t)}dt,
\end{equation*}
and\begin{equation*}
I_2=\int_{0}^{\infty}[\sigma(t)]^{(1-\theta)q}b^q(\sigma(t))\left(\int_{t}^{\infty}b_1^{q_1}(u)K^{q_1}(u,f)\frac{du}{u}\right)^{q/q_1}\frac{\sigma^\prime(t)}{\sigma(t)}dt.
\end{equation*}
We can compute that
$$\frac{\sigma^\prime(t)}{\sigma(t)}\approx\frac{t^{-1}b_1^{q_1}(t)}{\int_{t}^{\infty}b_1^{q_1}(u)\frac{du}{u}}.$$
First we show that  $I_2\approx\|f\|_{Y}^q$. Now $I_2\geq\|f\|_{Y}^q$ is a simple consequence of the fact that $u\mapsto K(u,f)$ is non-decreasing.  In order to establish the converse estimate $I_2\lesssim \|f\|_{Z}^q,$ we distinguish three cases: $q=q_1$, $q>q_1$ and $q<q_1.$ The case  $q=q_1$  simply follows from  Fubini's theorem. Next the case $q>q_1$ follows from Theorem \ref{HET2}, while the case $q<q_1$ follows from Theorem \ref{HET3}. Thus, it remains to show that $I_1\lesssim  \|f\|_{Y}^q$.  Again we distinguish three cases: $q=q_0$, $q>q_0$ and $q<q_0.$ The case  $q=q_0$  follows from  Fubini's theorem, while the case   $q>q_0$ follows from Theorem \ref{HET1} in view of the following estimate
$$\frac{b_0^{q_0}(t)}{\int_{t}^{\infty}b_0^{q_0}(u)\frac{du}{u}}\lesssim\frac{b_1^{q_1}(t)}{\int_{t}^{\infty}b_1^{q_1}(u)\frac{du}{u}}, \;\;t>0,$$
which is a simple consequence of our assumption that $\rho$ is increasing on $(0,\infty).$ As for the case  $q < q_0,$  we apply Theorem \ref{hmt1} with $\alpha=q/q_0,$ $h(t)=K(t,f)$, $w(t)=t^{-1}[\tilde{b}(t)]^q $, $\psi(t,u)=u^{-1}b_0^{q_0}\chi_{(0,t)}(u)$ and
$$v(t)=t^{-1}[\rho(t)]^{q(1-\theta)}b^q(\rho(t))[b_1(t)]^{q_1}\left(\int_{t}^{\infty}b_1^{q_1}(u)\frac{du}{u}\right)^{q/q_1-1}.$$  We observe that
 \begin{eqnarray*}
 \int_{0}^{\infty}\left(\int_{x}^{\infty}\psi(t, u)du\right)  ^{\alpha}w(t)dt &=& \int_{x}^{\infty}\left(\int_{x}^{t}b_0^{q_0}(u)\frac{du}{u}\right)^{q/{q_0}}[\tilde{b}(t)]^q\frac{dt}{t}\\
 &\leq&   \left(\int_{x}^{\infty}b_0^{q_0}(u)\frac{du}{u}\right)^{q/{q_0}} \int_{x}^{\infty}[\tilde{b}(t)]^q\frac{dt}{t}\\
 &\approx&\left(\int_{x}^{\infty}b_0^{q_0}(u)\frac{du}{u}\right)^{q/{q_0}}[\rho(t)]^{-\theta q}b^q(\rho(t)),
\end{eqnarray*}
and
 \begin{eqnarray*}
 \int_{x}^{\infty}v(t)dt &\gtrsim & [\rho(t)]^{q(1-\theta)}b^q(\rho(t)) \int_{x}^{\infty}[b_1(t)]^{q_1}\left(\int_{t}^{\infty}b_1^{q_1}(u)\frac{du}{u}\right)^{q/q_1-1}\frac{dt}{t}\\
 &\approx&   \left(\int_{x}^{\infty}b_0^{q_0}(u)\frac{du}{u}\right)^{q/{q_0}}[\rho(t)]^{-\theta q}b^q(\rho(t)).
 \end{eqnarray*}
Thus, $I_1\lesssim  \|f\|_{Y}^q$ holds.  The proof of the theorem is complete.
\end{proof}\\

\begin{remark}
{\em We have left out the cases $\theta=0$ and $\theta=1$ since in these cases no simplification takes place and the resulting   interpolation spaces involve the  $K$-interpolation spaces of type $\mathcal{L}$  and $\mathcal{R}$ (see, for instance, \cite{GOT}). We leave the details to the reader. }
\end{remark}

\begin{remark}
{\em We refer the reader to a recent reiteration formula \cite[Theorem 5.8]{AFF} which deals with the  case $q_0=q_1$ (without $K$-surjective assumption) for general weights (under certain appropriate conditions )  and for  ordered couples $(A_0, A_1)$ in the sense that  $A_1 \hookrightarrow A_0$.}
\end{remark}

The reiteration theorem corresponding to the limiting case $\theta_0=\theta_1=1$ reads as follows.
\begin{theorem}\label{ReT2} Let $0<q_0,q_1,q <\infty$, $0<\theta<1$,  $b_j \in SV_{1,q_j}$  (j=0,1),  and $b\in SV.$ Assume that $\eta$ is increasing on $(0,\infty)$ with $\lim\limits_{t \to 0^+}\eta(t)=0$ and $\lim\limits_{t \to \infty}\eta(t)=\infty.$  If  $q_0\neq q_1$, assume additionally    that the condition (\ref{HE3e1}) is met, while if $q_0=q_1$,   assume additionally that  the given couple $(A_0, A_1)$ is $K$-surjective and  that the following two-sided estimate holds:
 $$\frac{\eta^\prime(t)}{\eta(t)}\approx\frac{t^{-1}b_1^{q_1}(t)}{\int_{0}^{t}b_1^{q_1}(u)\frac{du}{u}}, \;\;t>0.$$  Put

\begin{equation*}
\hat{b}(t)=[\eta(t)]^{\theta}b(\eta(t))[b_1(t)]^{q_1/q}\left(\int_{0}^{t}b_1^{q_1}(u)\frac{du}{u}\right)^{1/q_1-1/q}.
\end{equation*}
Then

\begin{equation}\label{RTI}
\left(\bar{A}_{1,q_0; b_0}, \bar{A}_{1,q_1;b_1}\right)_{\theta,q;b}= \bar{A}_{1,q;\hat{b}}.
\end{equation}
\end{theorem}

\begin{remark}
    {\em The previous two reiteration theorems have already obtained in the special case when $b_j$ ($j=0,1$) are logarithmic functions (see \cite[Corollary 1]{D1}) or broken logarithmic functions (see \cite[Corollaries 7.8 and 7.11]{EOP}).}

\end{remark}

\begin{remark}
    {\em Let $0<\theta, \eta<1$ and $0<q\leq \infty$.  The  characterization of the interpolation spaces $$({\bar{A}}_{\theta,q;b_0}, {\bar{A}}_{\theta,q;b_1})_{\theta,r;b}$$ will involve the  $K$-interpolation spaces of type $\mathcal{L}$  and $\mathcal{R}$. We once again leave the elementary  details to the reader.  }

\end{remark}

\section{Concrete examples}
\subsection{Lorentz-Karamata spaces}
Let $(\Omega, \mu)$ be a $\sigma$-finite measure space. Let $f^\ast$ denotes the non-increasing rearrangement of a $\mu$-measurable function $f$ on $\Omega$ (see, for instance,  \cite{BS}).
\begin{definition} \cite{GOT} Let $0<p,q\leq \infty $ and $b\in SV.$ The Lorentz-Karamata space $L_{p,q;b}$ consists of all $\mu$-measurable functions $f$ on $\Omega$ such that the quasi-norm
$$\|f\|_{L_{p,q;b}}=\|t^{1/p-1/q}b(t)f^\ast(t)\|_{q,(0,\infty)}$$
is finite.
\end{definition}
For $b=\ell^\mathbb{A}$,  the spaces $L_{p,q;b}$ coincide with the spaces $L_{p,q;\mathbb{A}}$ from \cite{EO} and \cite{EOP}. When $b\equiv 1$, the spaces $L_{p,q;b}$ become the Lorentz spaces $L^{p,q}$, which coincide with the classical Lebesgue spaces $L^p$ for $p=q.$

We give an application of Theorem \ref{ReT2} to the interpolation of Lorentz-Karamata spaces $L_{p,q;b}$ in the  critical case when $p=\infty.$ To this end, we   characterize $L_{\infty,q;b}$ as limiting $K$-interpolation spaces for the couple $(L^1, L^\infty).$
\begin{lemma}\label{APL1} Let $0<q<\infty$ and $b\in SV_{1,q}.$ Then
$$L_{\infty,q;b}=(L^1, L^\infty)_{1,q;b}.$$

\end{lemma}
\begin{proof} Put $X=(L^1, L^\infty)_{0,q;b}$ and $Y=L_{\infty,q;b}$, and let $f\in L^1+L^\infty.$ Since (see \cite[Theorem $5.2.1$]{BL})
$$K(t,f;L^1,L^{\infty})=\int_{0}^{t}f^{\ast}(u)du,\;\; t>0,$$
it turns out that
$$
\|f\|_{X}=\left(\int_{0}^{\infty}t^{-q}b^{q}(t){\left(\int_{0}^{t}f^{\ast}(u)du\right)}^{q}\frac{dt}{t}\right)^{1/q}.
$$
Now the estimate $\|f\|_{X}\geq \|f\|_{Y}$ follows immediately in view of the fact that $f^\ast$ is non-increasing. On the other hand,  the converse estimate $\|f\|_{X}\lesssim  \|f\|_{Y}$ follows from Theorem \ref{HET1} (in the case  $q>1$), Theorem \ref{HET3+}  (in the case $q<1$) and Fubini's theorem (in the case $q=1$). The proof is complete.
\end{proof}\\

\begin{theorem}  Let $0<q_0,q_1,q <\infty$, $0<\theta<1$,  $b_j \in SV_{1,q_j}$  (j=0,1),  and $b\in SV.$  Assume that $\eta$ is increasing on $(0,\infty)$ with $\lim\limits_{t \to 0^+}\eta(t)=0$ and $\lim\limits_{t \to \infty}\eta(t)=\infty.$  If  $q_0\neq q_1$, assume additionally    that the condition (\ref{HE3e1}) is met, while if $q_0=q_1$,   assume additionally  that the following two-sided estimate holds:
 $$\frac{\eta^\prime(t)}{\eta(t)}\approx\frac{t^{-1}b_1^{q_1}(t)}{\int_{0}^{t}b_1^{q_1}(u)\frac{du}{u}}, \;\;t>0.$$  Then
\begin{equation*}
\left(L_{\infty,q_0;b_0}, L_{\infty,q_1;b_1}\right)_{\theta,q;b}=L_{\infty,q;\hat{b}},
\end{equation*}
where
\begin{equation*}
\hat{b}(t)=[\eta(t)]^{\theta}b(\eta(t))[b_1(t)]^{q_1/q}\left(\int_{0}^{t}b_1^{q_1}(u)\frac{du}{u}\right)^{1/q_1-1/q}.
\end{equation*}

\end{theorem}
\begin{proof} Take $A_0=L^1$ and $A_1=L^\infty$, and apply Theorem \ref{ReT2} to obtain
\begin{equation*}
\left((L^1, L^\infty)_{1,q_0;b_0}, (L^1, L^\infty)_{1,q_1;b_1}\right)_{\theta,q;b}=(L^1, L^\infty)_{1,q;\hat{b}}.
\end{equation*}
Now it remains to apply Lemma \ref{APL1}.
\end{proof}

\subsection{Generalized gamma spaces}

\begin{definition}(\cite{FFGKR})
Let $0<q,r\leq \infty$, $0<p<\infty$, $b\in SV_{0,q}$ and $w\in SV$  The generalized gamma space $\Gamma(r,q,p;b,w)=\Gamma(r,q,p;b,w)(\Omega)$ consists of all those real-valued Lebesgue measurable functions $f$ on $\Omega,$ for which the quasi-norm $$\|f\|_{\Gamma(r,q,p;b,w)}= \left\|t^{-1/q}b(t)\|\tau^{1/p-1/r}w(\tau)f^\ast(\tau)\|_{r, (0,t)}\right\|_{q, (0,\infty)}$$
is finite.
\end{definition}
\begin{theorem}\label{bst1}
Let $0<q_0\neq q_1<\infty$, $0<q, p, r <\infty$, $0<\theta<1$,  $b_j \in SV_{0,q_j}$  (j=0,1),  and $b,w\in SV.$  Assume  that the condition (\ref{HE1e1}) is met, and    assume that $\rho$ is increasing with  $\lim\limits_{t \to 0^+}\rho(t)=0$ and $\lim\limits_{t \to \infty}\rho(t)=\infty.$ Then
$$(\Gamma(r,q_0,p;b_0,w), \Gamma(r,q_1,p;b_1,w))_{\theta,q;b}= \Gamma(r,q,p;\tilde{b},w),$$
where
\begin{equation*}
\tilde{b}(t)=[\rho(t)]^{(1-\theta)}b(\rho(t))[b_1(t)]^{q_1/q}\left(\int_{t}^{\infty}b_1^{q_1}(u)\frac{du}{u}\right)^{1/q_1-1/q}.
\end{equation*}
\end{theorem}
\begin{proof} The proof follows from Theorem \ref{ReT1} and the following interpolation formula (for $j=0,1$)
$$\Gamma(r,q_j,p;b_j,w)=(L_{p,r;w}, L^\infty)_{0,q_j;b_j}.$$
\end{proof}

\subsection{ Homogeneous Besov spaces}
Let $E$ be a rearrangement invariant Banach function space on $\mathbb{R}^n$ as in \cite{BS}, and let $\omega_E(f,t)=\sup\limits_{|h|\leq t}\|\Delta_h f\|_E$ is the modulus of continuity of $f\in E$ (see, for example, \cite{BzK}).
\begin{definition} (\cite{BzK})
Let $0<q\leq \infty$ and $b \in SV_{0,q}$. The homogeneous Besov space $B^{0,b}_{E,q}$ consists of those functions $f\in E$ for which the semi-quasi-norm
$$\|f\|_{B^{0,b}_{E,q}}=\|t^{-1/q}b(t)(t)\omega_E(f,t)\|_{q,(0,\infty)}$$
is finite.
\end{definition}
It is well-known (see, for instance, \cite{BK}) that
$$
K(f,t;E,W^1E)\approx  \omega_E(f,t),\;\; t>0,
$$
where $W^1E$ is the Sobolev space built over $E$ with a norm $\|f\|_{W^1E}=\||D^1f|\|_E$. Here $|D^1f|=\sum\limits_{|\alpha|=1}|D^\alpha f|.$ Then it follows immediately that
\begin{equation}\label{bse1}
(E,W^1E)_{0,b,q}={B^{0,b}_{E,q}}.
\end{equation}
\begin{remark}
\em {We observe that Theorem \ref{ReT1} also holds when the compatible couple  quasi-normed spaces is replaced by a compatible couple of semi-quasi-normed spaces. }

\end{remark}
\begin{theorem}\label{bst1}
Let $0<q_0\neq q_1 <\infty$,  $0< q <\infty$, $0<\theta<1$,  $b_j \in SV_{0,q_j}$  (j=0,1),  and $b\in SV.$   Assume  that the condition (\ref{HE1e1}) is met, and    assume that $\rho$ is increasing with  $\lim\limits_{t \to 0^+}\rho(t)=0$ and $\lim\limits_{t \to \infty}\rho(t)=\infty.$ Then
$$(B^{0,b_0}_{E,q_0}, B^{0,b_1}_{E,q_1})_{\theta,q;b}=B^{0,\tilde{b}}_{E,q},$$
where
\begin{equation*}
\tilde{b}(t)=[\rho(t)]^{(1-\theta)}b(\rho(t))[b_1(t)]^{q_1/q}\left(\int_{t}^{\infty}b_1^{q_1}(u)\frac{du}{u}\right)^{1/q_1-1/q}.
\end{equation*}

\end{theorem}
\begin{proof} The proof follows from Theorem \ref{ReT1} and the interpolation formula $(\ref{bse1}).$
\end{proof}\\

\textbf{Acknowledgement}\\

The third author,  A. Gogatishvili started work on this project during his visit to J.M. Rakotoson to the Laboratory of Mathematics of the University of Poitiers, France, in April 2022. He thanks for his generous hospitality and helpful atmosphere during the visit.  He has been partially supported by the RVO:67985840 Czech Republic and by  Shota Rustaveli National Science Foundation of Georgia (SRNSFG), grant no: FR21-12353. \\

The authors are  very grateful  to Professor Fernando Cobos for pointing out the Definition \ref{DKS} to us.

\end{document}